\documentclass{article}

\usepackage{graphicx}

\usepackage[fleqn]{amsmath}
\usepackage{amsthm}
\usepackage{amssymb}
\usepackage{amsbsy} 
\usepackage{amsfonts}
\usepackage{mathrsfs}
\usepackage[all]{xy}
\usepackage{amstext}
\usepackage{amscd}
\usepackage[all]{xy}
\usepackage[dvips]{epsfig}
\usepackage{psfrag}
\usepackage{enumerate}
\usepackage{flafter}
\allowdisplaybreaks

\theoremstyle{plain}
\newtheorem{thm}{Theorem} 
\newtheorem{prop}[thm]{Proposition}

\theoremstyle{definition}

\newtheorem{rem}[thm]{Remark}

\def\Ker{\mathop{\mathrm{Ker}}\nolimits}

\newcommand{\lra}{\longrightarrow}
\newcommand{\ra}{\rightarrow}

\newcommand{\Z}{{\Bbb Z}}
\newcommand{\As}{{\rm As }}
\newcommand{\D}{{\cal D}_{g}}
\newcommand{\Dd}{{\cal D}_{2}}
\newcommand{\DD}{{{\cal D}^{\rm ns}_{g}}}
\newcommand{\DDd}{{{\cal D}^{\rm ns}_{2}}}
\newcommand{\M}{{\cal M}}

\newcommand{\pc}[2]{\mbox{$\begin{array}{c}
\includegraphics[scale=#2]{#1.eps}
\end{array}$}}

\title{ Finite presentations of centrally extended mapping class groups}
\author{Takefumi Nosaka}
\usepackage{fancyhdr}
\begin{document}
\maketitle
\begin{abstract} \noindent
We describe a finite presentation of $ \mathcal{T}_{g,r} $ for $g \geq 3$. 
Here $\mathcal{T}_{g,r} $ is the universal central extension of the mapping class group
of the surface of genus $g$ with $r$-boundaries.
We also investigate the cases of $g=2,3$, and give an application. 
\end{abstract}

\begin{center}
\end{center} 

\begin{center}
\normalsize
{\bf 
MSC2010 database}
\ \ \ 16S20, \ 37E30, \  	57M07 \ \
\end{center}

\large
\baselineskip=16pt
\section{Introduction and main results}\label{s1}

The mapping class group $\M_{g,r} $
of the surface $\Sigma_{g,r}$ of genus $g$ with $r$-boundaries
has been heavily studied in low-dimensional topology.
Especially, in quantum topology, Witten \cite{Wi} has made
a prophetic discovery that
the Chern-Simons quantum field theory of level $k$ (with Wilson loops)
produces 3-manifold invariants and (quantum) representations $V_k$ of $\mathcal{M}_{g,r} $ with an ``anomaly";
this prophecy has since been mathematically formulated in some cases (see, e.g., \cite{BHMV,GM,Koh}) and
the anomaly has been interpreted as a 2-framing \cite{Ati} or a $p_1$-structure \cite{BHMV}.
Because of the obstacles from by $p_1$-structures,
the space $V_k$ is not always 
some right module of $\M_{g,r}$, 
but that of 
a central extension over $\M_{g,r}$.

Such central extensions are sometimes considered to be complicated in a sense. In contrast,
the subjacent groups $\M_{g,r}$ have been widely analysed with
finite presentations (see, e.g., \cite{FM,Ge2}).
To see this, since $\M_{g,r} $ with $g \geq 3$ is perfect (see \cite{Kor}),
we can set up the central extension
\begin{equation}\label{tyuusing2} 0 \lra \Z \lra \mathcal{T}_{g,r} \xrightarrow{\ \mathrm{proj.}\ } \M_{g,r} \lra 0 \ \ \ \ \ \ \ {\rm (exact) }, \end{equation}
associated with the group {\it co}homology $H^2(\M_{g,r} ;\Z) \cong \Z$, which can be computed in a combinatorial way \cite[\S 6]{Kor}.
A 2-cocycle $\tau_{g}$ corresponding to a generator of the center $\Z$ has some difficulties, as mentioned in \cite{Ati}; 
the results of Meyer \cite{Meyer} and Turaev \cite{Turaev} indicate that
the quadruple $4 \tau_{g} $ can be geometrically described as a signature from the modular group $Sp(2g;\Z)$;
however, known formulations of $ \tau_{g}$ are algebraically a little complicated.
For example, Turaev formula of $ \tau_{g}$ \cite{Turaev} is
a signature 2-cocycle with a modification using a Maslov index;
furthermore, 
in \cite{GM},
the structure of $\mathcal{T}_{g,r} $ was described in terms of Lagrangian cobordisms. 
\subsection{Results; The cases of $g \geq 3$ and $r =0 , \ 1 .$}\label{ss1}
This paper explicitly describes a finite presentation of $\mathcal{T}_{g,r} $ for $g \geq 3, \ r \geq 0$.
To begin, the presentations 
with $r =0, \ 1$ are as follows: 
\begin{thm}\label{thm1} (I) Let $g \geq 3$. 
The central extension $ \mathcal{T}_{g,1} $ of $ \M_{g,1}$ in \eqref{tyuusing2} has a presentation with generators $c_0, \ c_1, \dots, c_{2g+1}, \mu$ and
the following relations: 
\begin{equation}\label{cet1} {\rm (Braid \ relation)}\ \ \ \ \ \ \ \ \ c_i c_j = c_j c_i, \ \ \ {\rm if}\ \ \ I(c_i,c_j)=0, \ \end{equation}
\begin{equation}\label{cet2} \ \ \ \ \ \ \ \ \ \ \ \ \ \ \ \ \ \ \ \ \ \ \ \ \ \ \ \ c_i c_j c_i= c_j c_ic_j , \ \ \ {\rm if}\ \ \ I(c_i,c_j)=1, \ \end{equation}
\begin{equation}\label{cet4} {\rm (3}\textrm{-}{\rm chain \ relation).}\ \ \ \ \ (c_1 c_2 c_3)^4 ( c_0 b_0)^{-1} = \mu, \ \ \ \ \ c_i \mu= \mu c_i, \end{equation}
\begin{equation}\label{cet5} {\rm (Lantern \ relation).}\ \ \ \ \ c_0 b_2 b_1 ( c_1 c_3 c_5 b_3)^{-1} =1, \end{equation}
where $I$ means the geometric intersection number in Figure \ref{tg261},
and the notation $b_0,b_1,b_2,b_3$ is common ones in the Wajnryb's presentation \cite{Waj}: 
Precisely, we have
\[ b_0 =(c_4 c_3 c_2 c_1 c_1 c_2 c_3 c_4)^{-1} c_0 (c_4 c_3 c_2 c_1 c_1 c_2 c_3 c_4),\]
\[ b_1 =(c_4 c_5 c_3 c_4 )^{-1} c_0 (c_4 c_5 c_3 c_4 ), \ \ \ \ \ \ b_2 =(c_2 c_3 c_1 c_2 )^{-1} b_1 (c_2 c_3 c_1 c_2 ),\]
\[ b_3 =(c_6 c_5 c_4 c_3 c_2 c_5^{-1}c_6^{-1} b_1 c_6 c_5 c_1^{-1} c_2^{-1}c_3^{-1}c_4^{-1})^{-1} c_0 (c_6 c_5 c_4 c_3 c_2 c_5^{-1}c_6^{-1} b_1 c_6 c_5 c_1^{-1} c_2^{-1}c_3^{-1}c_4^{-1}).\]

\noindent
(II)
Furthermore, concerning the closed surface of genus $g \geq 3$,
the group $ \mathcal{T}_{g,0} $ can be 
presented as above with adding the following commutator relation:
\begin{equation}\label{cet25} [(c_{2g} c_{2g-1} \cdots c_1 c_1 c_2 \cdots c_{2g} ), c_{2g+1}] =1, 
\end{equation}
where the bracket $[a,b]$ is the abbereviation of $aba^{-1}b^{-1}.$
\end{thm}
This presentation is a lift of
the Wajnryb presentation of $ \mathcal{M}_{g, r}$ \cite{Waj}.
To be precise,
his presentation can be made exactly as the quotient of the above presentation by adding the relation $\mu =1$; see also \cite[Theorem 5.3]{FM} for the detail.
Correspondingly, the symbols $c_i$ and $b_i$ in $ \mathcal{M}_{g, r} $
can be interpreted as Dehn twists along the respective curves
$\gamma_i $ and $\beta_i$ in Figure \ref{tg261}.



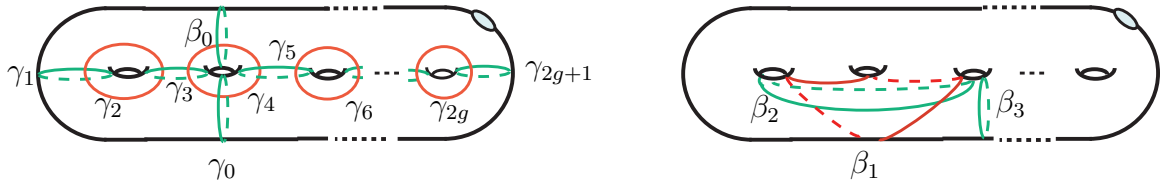
\begin{figure}[h]
$$
\begin{picture}(220,71)
\put(-112,34){\pc{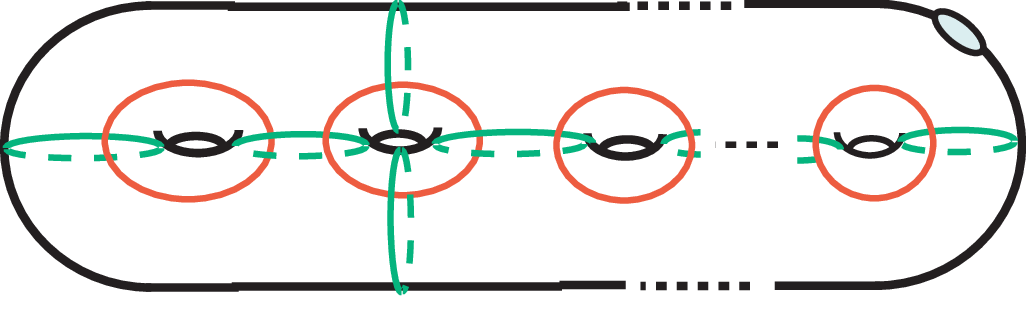}{0.36644466755}}

\put(132,34){\pc{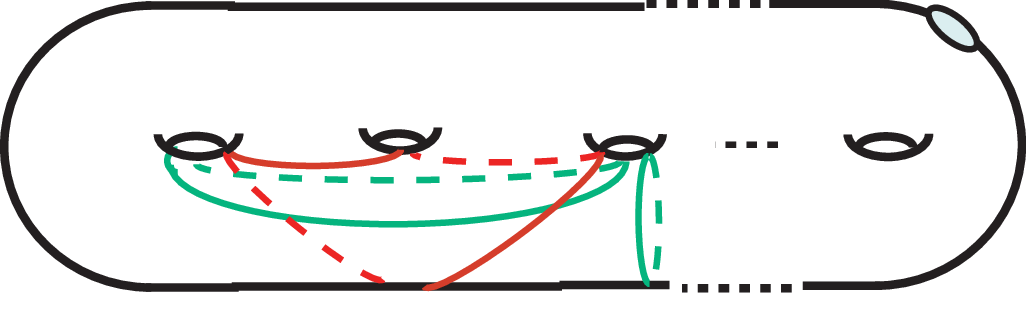}{0.36644466755}}

\put(-118,38){\large $\gamma_1$}
\put(-85,26){\large $\gamma_2$}
\put(-58,32){\large $\gamma_3$}
\put(-27,28){\large $\gamma_4$}
\put(10,25){\large $\gamma_{6}$}

\put(-19,48){\large $\gamma_5$}
\put(-51,53){\large $\beta_0$}
\put(-42,3){\large $\gamma_0 $}
\put(42,25){\large $\gamma_{2g}$}
\put(78,40){\large $\gamma_{2g+1} $}

\put(201,3){\large $\beta_1 $}
\put(163,24){\large$\beta_2$}
\put(256,26){\large $\beta_3$}

\end{picture}
$$

\vskip -1.225pc
\caption{\label{tg261} Generators of $\M_{g, *}$ with $g \geq 2$, $r\leq 1$, and the curves in the lantern relation with $g \geq 3$.}
\end{figure}

\vskip -2.88225pc

\






\subsection{Result II; Punctured cases with $g \geq 3$, and the genus two case.}\label{ss12}

Furthermore, we will examine the case $g \geq 3, \ r \geq 1$.
Following \cite{Ge2}, we call a triple $(i,j,k) \in \{1, \dots, 2g +r-2 \}^3 $ {\it good},
if it satisfies $ i\leq j \leq k $, or $j \leq k \leq i $ or $k \leq i \leq j$ without $i=j=k.$
Considering the closed curves in $\Sigma_{g,r}$ depicted in Figure \ref{gg},
let us state the expression using good indices: 
\begin{thm}[{cf. \cite[Theorem 1]{Ge2}}]\label{thm4} Let $g \geq 3$ and $r \geq 1$. 
The central extension $ \mathcal{T}_{g,r}$ admits a presentation
with generators $b, b_1, \dots, b_{g-1}, \ a_1,\dots, a_{2g+r -2}, \ \{ c_{i,j}\}_{ 1 \leq i,j \leq 2g+r-2, \ i \neq j} $
and $\mu$. Here, 
the relations are as follows:
\noindent
\begin{enumerate}[(i)]
\item ``Handles": $c_{2i, 2i \mp 1}=c_{2i \pm 1, 2i} $ for all $i$ with $ 1 \leq i \leq g-1$. 
\item ``Braids": $xy (xy^{-1})^{ I(x,y)} =yx$ holds for all $x,y$ among the generators with $I(x,y) \leq 1$, where $I(x,y)$ is the geometric intersection number of $x,y$ according to Figure \ref{gg}.
\item ``Stars": $ c_{i,j} c_{j,k} c_{k,i} = (a_i a_j a_k b)^3 \mu^{-1}$ for all good triples $ (i,j,k) $.
Here, we set $c_{\ell, \ell}=1$.
\item ``Centralization": $ [b, \mu]=[b_i, \mu]=[a_i, \mu]=[c_{i,j}, \mu]=1$.
\end{enumerate}
\end{thm}

In analogy to Theorem \ref{thm1}, this presentation is a lift of Gervais's presentation \cite{Ge2} of $\mathcal{M}_{g,r} $.
More precisely, the theorem 1 in \cite{Ge2} with $g \geq 1, \ r \geq 1 $ says that
the group with the presentation subject to $\mu=1 $ is isomorphic to $\M_{g,r} $.

\begin{figure}[h]
$$
\begin{picture}(220,109)
\put(-72,43){\pc{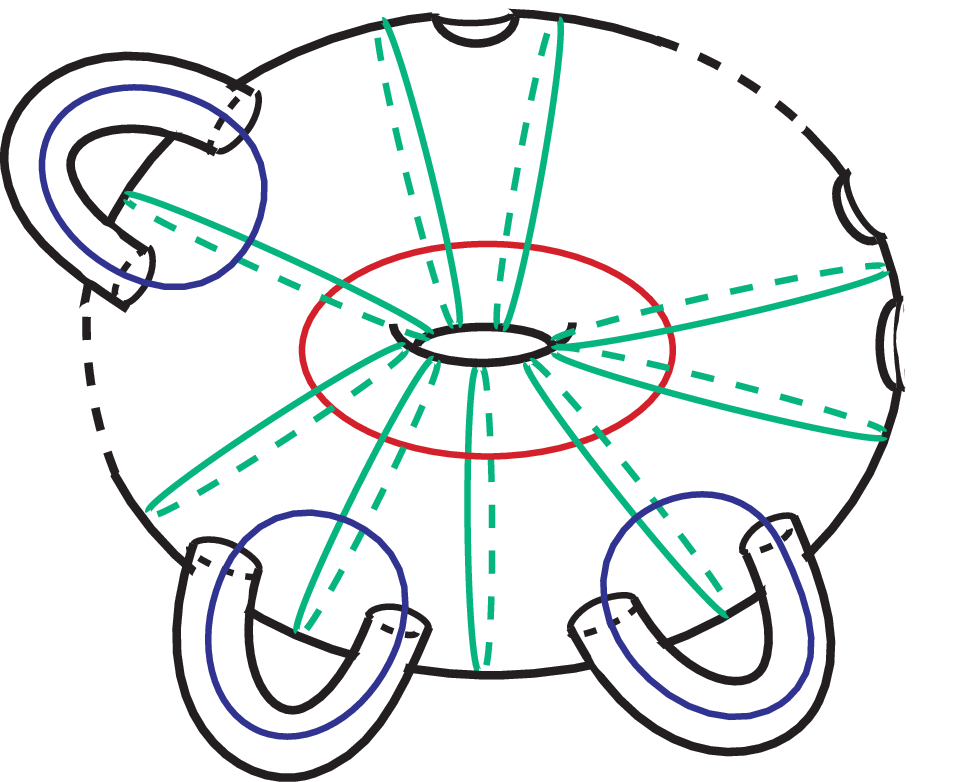}{0.32116644466755}}

\put(152,43){\pc{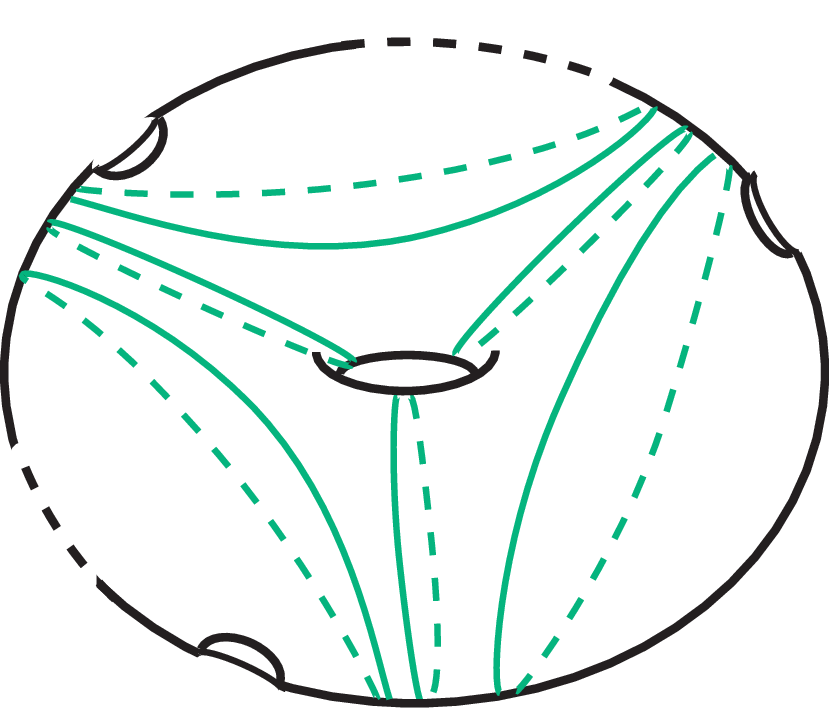}{0.32116644466755}}

\put(-48,38){\large $\alpha_5$}

\put(-25,32){\large $\alpha_4$}
\put(16,29){\large $\alpha_{2}$}
\put(-7,25){\large $\alpha_{3}$}

\put(-17,109){\large $\alpha_{2g-1}$}
\put(19,109){\large $\alpha_{2g} $}

\put(-26,70){\large $\alpha_{2g-2} $}
\put(-29,52){\large $\beta$}
\put(-77,61){\large $\beta_{g-1}$}
\put(-52,3){\large $\beta_2 $}
\put(65,3){\large $\beta_{1}$}
\put(72,36){\large $\alpha_{1} $}
\put(72,66){\large $\alpha_{2g+r-2} $}

\put(148,68){\large $\alpha_{j} $}
\put(262,88){\large $\alpha_{k} $}
\put(218,-13){\large $\alpha_{i} $}
\put(178,24){\large$c_{i,j}$}
\put(259,26){\large $c_{k,i }$}
\put(208,82){\large$c_{j,k}$}

\end{picture}
$$

\caption{\label{gg} Gervais's generators of $\M_{g, r}$ with $g \geq 2$, $r\geq 1$.}
\end{figure}
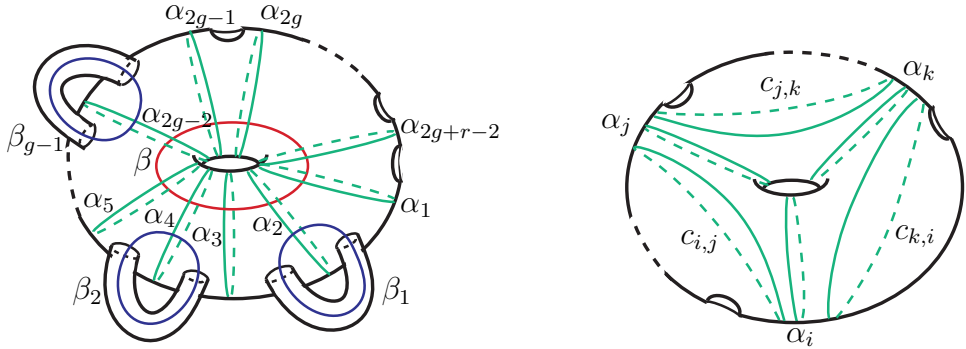

Finally, we focus on the case $g=2$.
Although $\mathcal{T}_{2}$ is not perfect,
we will give a presentation of a $\Z$-central extension $ \mathcal{T}_{2} $ of $\M_2$.
Here, $ \mathcal{T}_{2} $  is the central extension associated with a generator of 
$H^2(\M_2  ;\Z ) \cong H_1(\M_2 ) \cong \Z/10$.
\begin{thm}\label{thm2} 
The $\Z$-central extension $ \mathcal{T}_{2} $ of $\M_2$ has a presentation with generators $ c_1$, $c_2$, $c_{3}$, $c_{4}$ and $c_5 $.
Here, the relations are defined by the two preceding ones \eqref{cet1}, \eqref{cet2} and the followings: 
\begin{equation}\label{cet24}  (c_1 c_2 c_3)^4 c_5^{-2}  = c_{5} c_{4} c_3 c_2 c_1 c_1 c_2c_{3} c_{4} c_5 , \end{equation}
\begin{equation}\label{cet25666} [c_{5} c_{4} c_3 c_2 c_1 c_1 c_2c_{3} c_{4} c_5 ,c_5 ]=1 . \end{equation}
\end{thm}
Although we do not focus on the remaining case of genus one, the presentation of $\mathcal{T}_{1,r}$ was already studied (see \cite{GM}).
In conclusion, our results are summarized to that
most of the central extensions $ \mathcal{T}_{g,r}$
can be dealt with concretely as finitely presented groups. 

\

Finally, we list several little applications in some topics.
First, 
the results are useful for checking whether a linear representation $\rho: \mathcal{T}_{g,r} \ra GL_n (\mathbb{C})$ is well-defined
or not: For instance, if $r=0$ or $r=1$, according to Theorem \ref{thm1}, it is necessary to check that the lantern relation and the 3-chain relation
are sent to $ c\cdot \mathrm{id}_{ \mathbb{C}^n}$ for some $c\in \mathbb{C}.$
Furthermore, the theorems might be useful
in normalizing (constant factors of) Kohno's description \cite{Koh} on 3-manifold invariants and in computing concretely Lefschetz fibration invariants in \cite{Nos2}, which are constructed from $\rho_{\mathcal{G},k} $.
In addition, the presentation of $ \mathcal{T}_{g,r}$ seems compatible with the setting in the chart description \cite{EHKT}. 

\section{Proof of the theorems.}\label{ss2092}
This section 
is dedicated to proving the theorems stated above.
Here, we assume that the reader has basic knowledge of group cohomology and the mapping class group (see \cite[\S 1-5]{FM} and \cite{Kor}).

We briefly explain an outline of the proofs. 
First, we review the group ``$\As (\DD)$" that is introduced in the
study of the TQFT \cite{MR}; we show an isomorphism
$\As (\DD) \cong \Z \times \mathcal{T}_{g,0}$ for $g \geq 3$; see Theorem \ref{thm3}. As a corollary, we verify that
the group presented in Theorem \ref{thm1} (II) is isomorphic to $\As (\DD)$ without the $\Z$-factor.
In parallel, concerning Theorem \ref{thm2}, a similar discussion is applicable to the genus two case.
After that, we show Theorem \ref{thm1} (I) by using the universality of the central extensions and
Harer stability \cite{Har}. 
Next, we will give a proof of Theorem \ref{thm4} by induction on $r \geq 1$.
Actually, the group presented in Theorem \ref{thm4} for $r=1$ is shown to be
isomorphic to that in Theorem \ref{thm1}, and
the proof for $r \geq 2$ is similarly done by using a result of 
Harer stability \cite{Har}.


\subsection{Preliminaries: the associated group}\label{s3}
To accomplish the outline, we start by introducing terminologies
and state a key proposition (Theorem \ref{thm3}).
First, denoting $ \M_{g,0} $ by $\M_g $ for short, we set the following three subsets: 
\begin{equation}\label{defdg}\D:= \{ \ \tau_{\alpha } \in \M_{g} \ | \ \alpha \textrm{ is a (unoriented) simple closed curve} \ \gamma \ {\rm in \ } \Sigma_g \ \} , \end{equation}
\begin{equation}\label{defdg11}\DD:= \{ \ \tau_{\alpha } \in \D \ | \ \alpha \textrm{ is a non-separating simple closed curve} \ \gamma \ {\rm in \ } \Sigma_g \ \} , \end{equation}
\[ \D^{(k)}:= \{ \ \tau_{ \alpha} \in \D \ | \ \mathrm{\ The \ complement \ } \Sigma_g \setminus \alpha \mathrm{\ is \ homeomorphic \ to \ } \Sigma_{k,1} \sqcup \Sigma_{g-k,1} \ \}, \]
where the symbol $\tau_{\alpha}$ is the (positive) Dehn twist along $\alpha$.
Next, for $Z=\D $ or $Z=\DD$, we will analyze the group $\mathrm{As}(Z)$, which is considered in \cite{MR}.
Here $\mathrm{As}(Z)$ is defined to be the abstract group generated by symbols $e_z$ with $z \in Z$ subject to the relation $e_{w^{-1}zw }= e_{w}^{-1}e_z e_w $
and is called {\it the associated group}.
Note that the inclusion $\D\hookrightarrow \M_{g }$ gives rise to a group epimorphism
$\mathcal{E}: \As(Z) \ra \M_{g}$ by definition,
and we obtain 
the equality
\begin{equation}\label{rel} g e_{z} g^{-1} = e_{ \mathcal{E} (g e_{ z} g^{-1}) } \in \As(Z), \ \ \ \ \ {\rm for \ any \ \ } \ z \in Z, \ \ g\in \As(Z) ,
\end{equation}
which is easily verified by induction on the word length of $g$.
The reader should keep in mind this equality, since we will use \eqref{rel} in several times.
For example, as a result of \eqref{rel}, the kernel $\Ker ( \mathcal{E}) $ is contained in the center.
In summary, we have
\begin{equation}\label{bba} 0\lra \Ker (\mathcal{E}) \lra \As(Z) \stackrel{ \mathcal{E}}{\lra} \M_{g} \lra 0 , \ \ \ \ \ \ ({\rm central \ extension}).\end{equation}
In addition, we can explicitly determine the central extension $\As(Z)$ as follows:
\begin{thm}\label{thm3} {\rm (cf. \cite[Theorem C]{Ge})}
If $ g \geq 3$, there are isomorphisms $\As(\D) \cong \mathcal{T}_{g,0} \times \Z^{[g/2]+1} $
and $\As(\DD) \cong \mathcal{T}_{g,0} \times \Z $.

\end{thm}
\noindent
Here, we shall refer to \cite[Theorem C]{Ge} which claimed the same statement on $\As(\DD) $
and discussed an infinite presentation of $ \mathcal{T}_{g,r}$ from the $p^1$-structure.
But, his proofs
used functoriality with respect to lantern relations of $\widetilde{\M}_g$ \cite[Theorems C, 4.1]{Ge}, and
might contain a little gap with $g=3$.
In contrast, the proofs in this paper use neither combinatorial computation as in \cite{Ge}
nor the Maslov index as in \cite{Ati,GM,Turaev},
but only basic knowledge of group cohomology of degree 2. 



In addition, let us describe the concepts of lantern relations in $\As(\D)$. 
If $g \geq 3$, consider two elements of the form
$$ \kappa_{3\textrm{-}\mathrm{chain }} := (e_{c_1}e_{c_2}e_{c_3})^4 e_{c_0}^{-1}e_{b_0}^{-1}, \ \ \ \ \ \kappa_{\rm lantern} := e_{c_1}^{-1}e_{c_3}^{-1}c_{c_5}^{-1}e_{b_{3}}^{-1} e_{c_0}e_{b_2}e_{b_1} \in \As(\D), $$
where $b_i$ and $ c_i$ are the respective Dehn twists of the curves $\beta_i$ and $\ \gamma_i$ in Figure \ref{tg261};
if $g=2$, we define $\kappa_{3\textrm{-}\mathrm{chain }} $ to be $(e_{c_1}e_{c_2}e_{c_3})^4 e_{c_5}^{-2} $ 
and $\kappa_{\rm lantern} $ to be $\mathrm{1}_{\As (\D)}$.
As is well-known, $ \mathcal{E} (\kappa_{3\textrm{-}\mathrm{chain }} )$ and $ \mathcal{E} (\kappa_{\rm lantern} ) $ are the identity in $\mathcal{M} _g$,
which are commonly called the {\it 3-chain relation} and the {\it lantern relation}, respectively.
Furthermore, for $k <g/2$, set up seven curves $\alpha_k , \beta_k, \gamma_k, \delta_k, x_k,y_k,z_k$ in $\Sigma_g$ illustrated in
Figure \ref{lanternfig}, and
consider the product
$$\mathcal{L}_{k}:= e_{\tau_{\alpha_k} }^{-1} e_{\tau_{\beta_k}}^{-1} e_{\tau_{\gamma_k}}^{-1} e_{\tau_{\delta_k}}^{-1} e_{\tau_{x_k}} e_{\tau_{y_k}}e_{\tau_{z_k}} \in \As (\D). $$
The lantern relations in $\M_g $ tell us that these $ \kappa_{\rm lantern} $ and $\mathcal{L}_{k} $ lie in $\Ker (\mathcal{E} )$.

In addition, for $\dagger = 1,2, \dots, [g/2]$ or $\dagger = {\rm ns}$,
we define a homomorphism $ \epsilon_{\dagger} : \As (\D) \ra \Z $ by setting
$ \epsilon_{\dagger} ( e_x)=1 \in \Z $ if $ x \in \D^{(\dagger)}$ and by setting $ \epsilon_{\dagger} ( e_x)=0 \in \Z $ otherwise. 
Since the orbit decomposition of the conjugate action $ \D \curvearrowleft \M_g $ is presented as
$\D = \DD \cup \bigl( \sqcup_{0 < k \leq g/2}\D^{(k)} \bigr) $, 
it follows from \eqref{rel} that
the sum $(\bigoplus \epsilon_j) \oplus \epsilon_{\rm ns}$ gives an abelianization $ H_1(\As (\D) ) \cong \Z^{[\frac{g}{2}]+1 } $.
\vskip 2.15pc
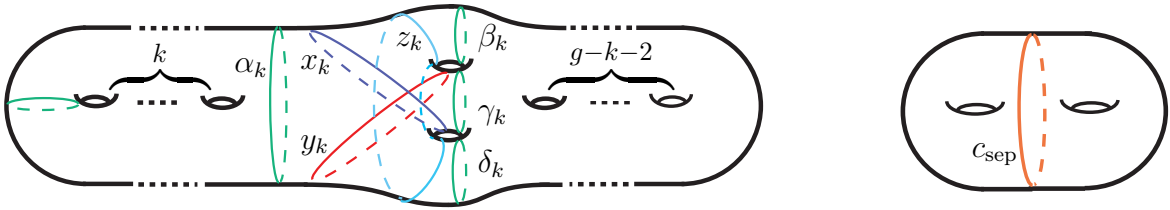
\begin{figure}[htpb]
\begin{picture}(400,60)
\put(100,50){\large $\alpha_k $}
\put(125,51){\large $x_k $}
\put(125,21){\large $y_k$}

\put(161,60){\large $z_k $}

\put(192,60){\large $\beta_k $}
\put(192,32){\large $\gamma_k $}
\put(192,12){\large $\delta_k$}

\put(379,18){\large $c_{\rm sep} $}

\put(53,43){\Large $\overbrace{ \ \ \ \ \ \ \ \ }^{k}$ }
\put(219,43){\Large $\overbrace{\ \ \ \ \ \ \ \ \ \ }^{g-k-2} $ }

\put(8,33){\pc{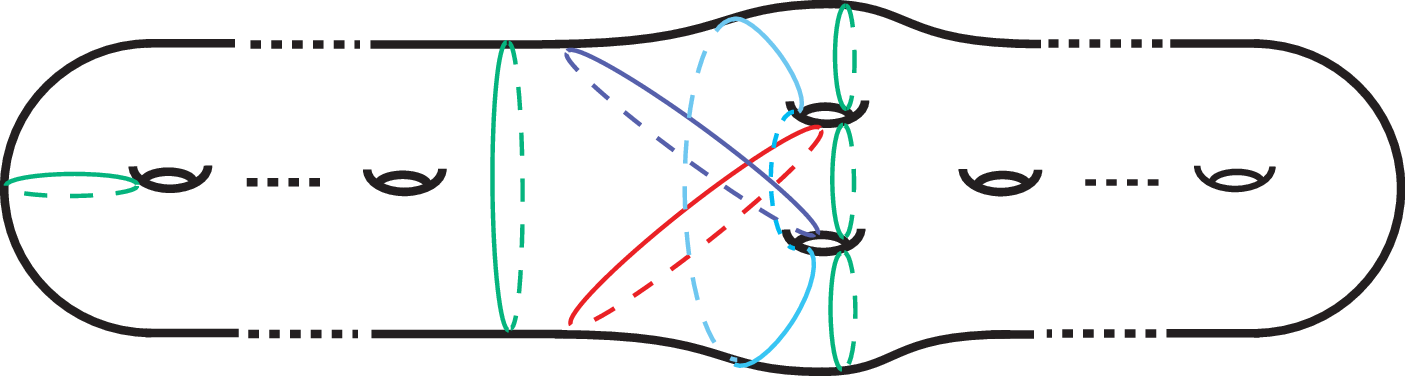}{0.424}}

\put(347,28){\pc{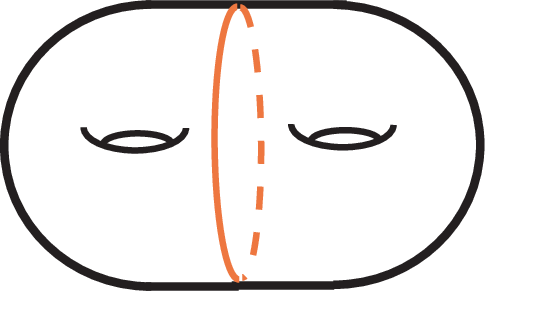}{0.433824}}
\end{picture}
\caption{\label{lanternfig} The $k$-th lantern relation $\mathcal{E}(\mathcal{L}_k ) $ in $\M_g$ with $g>2, \ k\geq 1 $, and the curve $c_{\rm sep}$ with $g=2$.}
\end{figure}
Furthermore, let us review \cite[Theorem 4.2]{EN}.
Notice from \eqref{rel} that, if a tuple $\mathbf{z}=(z_1, \dots, z_m) \in (\D)^m$ satisfies
$ z_1\cdots z_m=1_{\M_g}$, the product $ e_{z_1} \cdots e_{z_m} \in \As(\D )$
lies in the central kernel $\Ker(\mathcal{E} : \As(\D) \ra \M_g ) $.
Furthermore, we can construct a closed oriented 4-manifold $ E_{ \mathbf{z}}$ from
such a tuple $\mathbf{z} \in (\D)^m$, which is called {\it a Lefschetz fibration (over $S^2$)}; see, e.g., \cite{EN} for the definition.
Inspired by \cite{Ge} and the Hopf theorem on $H_2$,
Endo and Nagami \cite[Definition 3.3 and Proposition 3.6]{EN} constructed a homomorphism $I_g:\Ker(\mathcal{E} ) \ra \Z$ that enjoys the following property:
\begin{thm}[{\cite[Theorem 4.2 and Propositions 3.10-3.12]{EN}}]\label{ENresult}
The homomorphism $I_g $ satisfies 
$$I_g (e_{z_1} \cdots e_{z_m})= \sigma_{ \mathbf{z}}  \in \Z , $$
for any $m$-tuple $(z_1,\dots, z_m) \in (\D)^m$ with $z_1\cdots z_m =1_{\M_g} \in \M_g$.
Here $m_{\rm ns} \in \Z$ is the number $\{ j \ | z_j \in \DD \}$,
and $ \sigma_{ \mathbf{z}} $
is the signature of the associated 4-manifold $ E_{ \mathbf{z}}$. 

Moreover,
$I_g (\kappa_{3\textrm{-}\mathrm{chain }} )=-6$; Furthermore, if $g \geq 3$, then $ I_g ( \kappa_{\rm lantern} ) = I_g ( \mathcal{L}_k )= 1$.
\end{thm}
\subsection{Proofs of theorems for $g\geq 3$. }\label{s311}
We begin by proving Theorem \ref{thm3}, and show Theorems \ref{thm1}--\ref{thm4}.
Here, we should set up the epimorphism
\begin{equation}\label{rrr} \theta_r : \M_{g,r} \lra \M_{g,r-1},
\end{equation}
induced by gluing a disc to the boundary component of $\Sigma_{g,1}$.
\begin{proof}[Proof of Theorem \ref{thm3}]
For the proof, we now analyze the group $\As (\D)$ with $g \geq 3$,
as a special tool in quandle theory similar to \cite[\S\S 2--3]{N1}.
Consider a homomorphism $\mathfrak{s}_j :\Z \ra \As (\D) $
which sends $ n$ to $ (\mathcal{L}_{j})^n.$
These $\mathfrak{s}_j $ is a section of $\epsilon_j$, and
the image is contained in the center $\Ker (\mathcal{E})$ because of \eqref{rel}.
Similarly, take the map $\mathfrak{s}_{\rm ns} :\Z \ra \As (\D) $
which sends $ n$ to $ (e_{c_0} e_{b_2} e_{b_1} ( e_{c_1} e_{c_3}e_{ c_5}e_{ b_3})^{-1})^{-n}.$
Therefore, the semi-direct product structure associated with $ \mathfrak{s}_j $ and $\mathfrak{s}_{\rm ns} $ is trivial;
hence, we have the decomposition $\As (\D)
\cong \tilde{\mathcal{M}}_g\times \Z^{[\frac{g}{2}]+1 }$ for some
central extension $\tilde{\mathcal{M}}_g$ of $\mathcal{M}_g$.
By K\"{u}nneth formula on $H_1$, this $\tilde{\mathcal{M}}_g$ is perfect.

Hence, it is sufficient to show $ \tilde{\mathcal{M}}_g \cong \mathcal{T}_{g,0}$.
To this end, notice from the infraction-restriction exact sequence (cf. \eqref{OO} below)
that 
the kernel of $\tilde{\mathcal{M}}_g \ra \M_g$ is a quotient of $H_2(\M_g)$,
since $ \tilde{\mathcal{M}}_g$ is perfect.
Here, we should notice that the kernel contains $\Z$.
Actually, it follows from Theorem \ref{ENresult} that
the sum of the homomorphisms $ (I_g \oplus \epsilon_{\rm ns}) \oplus (\bigoplus \epsilon_j) : \Ker( \mathcal{E} ) \ra \Z^{ [\frac{g}{2}]+2}$ is
of order 4 at most.

We will show $ \tilde{\mathcal{M}}_g \cong \mathcal{T}_{g,0}$ as a result of the preceding claim. 
First, if $g\geq 4$, the kernel of $\tilde{\mathcal{M}}_g \ra \M_g$ must be $H_2(\M_g) \cong \Z$;
hence, the universality of the central extensions implies $\tilde{\mathcal{M}}_g \cong \mathcal{T}_{g,0} $ as desired.

Next, we address the case $g=3$.
We should refer the fact $H_2(\M_{3,1} ) \cong H_2(\M_3 ) \cong \Z \oplus \Z/2$ \cite[Corollary 4.10]{Sakasai};
Thus, the kernel of the projection $q: \tilde{\mathcal{M}}_3 \ra \M_3 $ 
is either $\Z$ or $\Z/2 \oplus \Z $. 
We will show that the kernel $\Ker (q)$ is $\Z $.
For this, consider the pullback of $q : \tilde{\mathcal{M}}_3 \ra \M_3 $ and $\theta_1 $, 
and denote it by $\tilde{\mathcal{M}}_{3,1}  $
Since the induced map $\theta_1: H_2(\M_{3,1} ) \ra H_2(\M_3 ) $ is known to be isomorphic \cite[\S 4]{Sakasai}, it is enough for the proof of $ \Ker (q)\cong \Z$ to show that $\tilde{\mathcal{M}}_3 $
is not the universal central extension of $ \M_3$.  
Consider the group with presentation
$$\mathcal{M}_{3,1}^{\rm pre} := \langle \ c_0, c_1,  \dots, c_{7} \ | \ \textrm{the relations} \ \eqref{cet1}, \ \eqref{cet2}, \mathrm{ \ and \ }  \eqref{cet4} \ \rangle. $$
The Wajnryb presentation \cite{Waj} implies that $\mathcal{M}_{3,1}^{\rm pre} $ is a central $\Z$-extension over $\mathcal{M}_{3,1} $ and the kernel is $\Z = \{ \mu^{m} \}_{m\in \Z} $, 
and that $\mathcal{M}_{3,1}^{\rm pre} $ is perfect by the lantern relation \eqref{cet5}.
Then, we shall notice that the canonical map $ \mathcal{M}_{3,1}^{\rm pre} \ra \mathcal{M}_{3,1}$ which sends $ c_i$ to $e_{c_i}$ is surjective, since every generator $e_{\tau_{\alpha} }$ is conjugate to $c_0 $. Hence, if $\tilde{\mathcal{M}}_{3,1} $
is the universal central extension of $ \M_{3,1}$, there is such no map from any central extension of $\M_{3,1} $ with fiber $\Z$. Hence, $\tilde{\mathcal{M}}_{3,1} $ is not universal. Hence,  the kernel $\Ker(q) $ is $\Z$, leading to $\tilde{\mathcal{M}}_3  \cong \mathcal{T}_3 $. 


Finally, we show $\As (\DD) \cong \mathcal{T}_{g,0} \times \Z$ with $g\geq 3 $.
Notice from \eqref{rel} that $ \epsilon_{\rm ns}:\As (\DD) \ra \Z$ is the abelianization. 
Similarly, the image of $\mathfrak{s}_{\rm ns} :\Z \ra \As (\DD) $
is contained in the center $\Ker (\mathcal{E})$. Therefore, the semi-direct product structure is trivial, i.e.,
$\As (\DD)\cong \tilde{\mathcal{M}}_g'\times \Z$ for a central extension $\tilde{\mathcal{M}}_g'$ of $\mathcal{M}_g$.
Thus, it is enough to show $\tilde{\mathcal{M}}_g' \cong \mathcal{T}_{g,0} $.
For this, the map $\As (\DD) \ra \As(\D) $ induced by the inclusion $ \DD \ra \D $
implies $\tilde{\mathcal{M}}_g'\times \Z  \ra \mathcal{T}_{g,0}\times \Z^{[\frac{g}{2}]+1 }$ over $\mathcal{M}_g $. Note that the image of $\tilde{\mathcal{M}}_g'$ is contained in $\mathcal{T}_{g,0}$, since $\tilde{\mathcal{M}}_g'$ is also perfect.
Thus, the universality of $\mathcal{T}_{g,0}$ immediately lead to the desired $\tilde{\mathcal{M}}_g' \cong \mathcal{T}_{g,0}$.
\end{proof}
We are now in a position to prove Theorem \ref{thm1}.
Notice from Theorem \ref{ENresult} that
\begin{equation}\label{rrr2}I_g(\kappa_{3\textrm{-}\mathrm{chain }} \kappa_{\rm lantern}^{10} )=4, \ \ \ \ \ \epsilon_{\rm ns} (\kappa_{3\textrm{-}\mathrm{chain }} \kappa_{\rm lantern}^{10} )=0.\end{equation}
Since the cokernel of $(I_g \oplus \epsilon_{\rm ns}) \oplus (\bigoplus \epsilon_j) $ is $\Z/4$ as in the proof above,
this $\kappa_{3\textrm{-}\mathrm{chain }} \kappa_{\rm lantern}^{10} $ is a generator of the center $\Z= \Ker (\mathcal{T}_{g,0} \ra \mathcal{M}_g )$ in \eqref{tyuusing2}.
\begin{proof}[Proof of Theorem \ref{thm1}.]
Let $r=0$ or $1$, let $\mathcal{G}_{g,r}$ be the group with the presentation given in Theorem \ref{thm1},
and let $ q_r :\mathcal{G}_{g,r} \ra \mathcal{M}_{g,r} $ be the quotient map by adding the relation $\mu= 1.$
Noting the relation \eqref{cet4}, the map $ q_r $ is a central extension with fiber $\Z$.

We first show (II). 
Using \eqref{rel},
we can verify that the correspondence
$$ c_i \longmapsto e_{c_i } \kappa_{\rm lantern}, \ \ \ \ \ \mu \longmapsto \kappa_{3\textrm{-}\mathrm{chain }} \kappa_{\rm lantern}^{10} $$
defines a homomorphism $ \psi : \mathcal{G}_{g,0} \ra \As (\DD)$ over $\M_g .$
The image is contained in $ \mathcal{T}_{g,0}  \cong  \Ker (\epsilon_{\rm ns} : \As (\DD) \ra \Z )$ by definition.
Hence, since $\mathcal{T}_{g,0} $ is universal modulo torsion subgroup,
$ \psi$ must be an isomorphism $ \mathcal{G}_{g,0} \cong \mathcal{T}_{g,0}$ as required.

\noindent

Next, we show (I). 
From the definition of \eqref{rrr}, we have a commutative diagram:
\begin{equation}\label{ddd}{\normalsize
\xymatrix{
0 \ar[r] & \Z \ar[rr] \ar[d] \ar[rr]&& \mathcal{G}_{g,1} \ar[rr]^{q_1}\ar[d]^{\rm proj.} & & \mathcal{M}_{g,1} \ar[r] \ar[d]^{\theta_1 } &0 & ({\rm central \ extension}) \\
0 \ar[r] & \Z \ar[rr] && \mathcal{G}_{g,0} \ar[rr]^{q_0} & & \mathcal{M}_g \ar[r] & 0 & ({\rm central \ extension}) .
}}
\end{equation}
From the definitions of $\mathcal{G}_{r,*} $, the left vertical map is isomorphic.
It is known as the Harer stability with $ g \geq 3$ (see \cite[\S 6]{Kor}) that the right induced map
$ \theta_1^*: H^2(\M_{g} ;\Z) \ra H^2(\M_{g,1} ;\Z)$ is an isomorphism on $\Z$.
Since $ \mathcal{G}_{g,0} \cong \mathcal{T}_{g,0} $, the universality of the central $\Z$-extensions
implies the desired isomorphism $ \mathcal{G}_{g,1} \cong \mathcal{T}_{g,1} $.
\end{proof}

Now, let us turn into proving Theorem \ref{thm4} for the punctured groups.
Let us denote by $ \mathcal{G}_{g,r}'$ the group with the presentation given in Theorem \ref{thm4}.
Recalling the quotient $ \mathcal{G}_{g,r}'/ \langle \mu=1 \rangle \cong \M_{g,r}$ as a result of \cite[Theorem 1]{Ge2},
we see from the relation (iv) that the projection $\mathcal{G}_{g,r}' \ra \M_{g,r} $ is a central $\Z$-extension.

\begin{proof}[Proof of Theorem \ref{thm4}]
As the first step of the induction on $r$,
we let $g \geq 3, \ r=1$ and will observe a diagram similar to \eqref{ddd}.
Consider the map $ q_1 : \mathcal{G}_{g,1}' \ra \mathcal{M}_{g,1} $
which takes each generator $\alpha $ of $ \mathcal{G}_{g,1} '$ to the corresponding Dehn twist $ \tau_{\alpha }$.
It immediately follows from \cite[Lemma 5]{Ge2} that
the kernel of $q_1 $ is generated by only $ \mu =1 $, 
and that $q_1$ is a central $\Z$-extension from the definition of $\mu$.
Furthermore, using the homomorphism $\theta_1 $ in \eqref{ddd}, 
consider the correspondence from $\mathcal{D}_g $ to $\As (\D)$ defined by
\begin{equation}\label{ddddd}
\theta_1(\gamma) \longmapsto e_{\theta_1 (\gamma)}  \kappa_{\rm lantern}^{\epsilon_{\rm ns} (\theta_1 (\gamma))} \mathcal{L}_1^{\epsilon_{1} (\theta_1 (\gamma)) } \cdots \mathcal{L}_{[g/2]}^{\epsilon_{[g/2]} (\theta_1 (\gamma)) } , \ \ \ \ \ \theta_1(\mu ) \longmapsto \kappa_{3\textrm{-}\mathrm{chain }} \kappa_{\rm lantern}^{10} .
\end{equation}
Here, $ \gamma$ runs over the generators in Theorem \ref{thm4}, and $\mathcal{L}_k $ is the central element defined in \S \ref{s3}.
Notice from \cite[Proposition 3.13]{EN} that
the homomorphism $I_g :\Ker (\mathcal{E}) \ra \Z $ in Theorem \ref{ENresult} (resp. $\epsilon_{\rm ns}$) sends
the star relation (iii) to $ 5- N_{i,j,k} $ (resp. $ 9 + N_{i,j,k}$),
where $N_{i,j,k}$ is the cardinality of $\{ c_{i,j} , \ c_{j,k}, \
c_{k,i} \ | \ c_{x,y} \ {\rm is \ separating.}\ \}$.
Hence, compared with \eqref{rrr2}, the correspondence \eqref{ddddd} defines a map $ \mathcal{G}_{g,1}' \ra \As (\D)$ as a centrally extended homomorphism over $\theta_1 $.
Note that
the image is $\Ker ( \oplus \epsilon_{\dagger}: \As(\D) \ra \Z^{[g/2]+1 }) \cong \mathcal{T}_{g,0} $ by definition.
Hence, $\mathcal{G}_{g,1} '$ must be $\mathcal{T}_{g,1}$ by a diagram chasing similar to \eqref{ddd} and the universality of central $\Z$-extensions

Finally, we now complete the proof with $r \geq 2$.
Consider the canonical surjection $ p_{r}: \mathcal{G}_{g,r}' \ra \mathcal{G}_{g,r-1}'$ obtained from the
presentations.
Then, the quotient $p_r $ modulo $ \mu=1 $ is identified with $ \theta_r: \M_{g,r} \ra \M_{g,r-1 } $ in \eqref{rrr}.
In addition, consider the injection
$\iota_r: \M_{g,r-1} \ra \M_{g,r }$
induced by gluing a two holed disc to a boundary component of $\Sigma_{g,r -1 }$.
From the Harer stability (see \cite{Kor}), the induced map $ \iota_{r}^*: H^2(\M_{g,r } ;\Z) \ra H^2( \M_{g,r-1 } ;\Z)$ is known to be an isomorphism on $\Z$.
Furthermore, since $ \theta_{r} \circ \iota_r = \mathrm{id}$,
the induced map $ \theta_{r}^* : H^2(\M_{g,r-1 } ;\Z) \ra H^2( \M_{g,r } ;\Z)$ is an isomorphism on $\Z$.
Consequently, the surjection $ p_{r}$ induces an isomorphism on $H^2$.
Hence, since $ H^2( \mathcal{G}_{g,r-1}';\Z) \cong 0 $ by induction on $r$,
we have $ H^2( \mathcal{G}_{g,r}';\Z) \cong 0 $.
In conclusion,
the central $\Z$-extension $ \mathcal{G}_{g,r}' $ is also universal up to torsion subgroup, that is, $\mathcal{G}_{g,r } ' \cong \mathcal{T}_{g,r}$ as desired.
\end{proof}


\subsection{The case of two genus}\label{s3777}
This subsection is devoted to giving the proof with $g=2$. 

As a preliminary, we will describe the extension $\mathcal{T}_2$ in terms of $\mathcal{D}_2$. 

For this, we begin by showing the isomorphism \eqref{dsds} below. 
Consider the homomorphism $\mathfrak{s} $ defined by setting
\[ \mathfrak{s}:\Z \lra \As(\Dd ); \ \ \ n\longmapsto \bigl( (e_{c_1 } e_{c_2} )^{6}e_{\tau_{c_{\rm sep}}}^{-1}\bigr)^n , \]
Here and $ c_{\rm sep}$ are the curve 
described in Figure \ref{lanternfig}, respectively.
Then the 2-chain relation in $\M_2 $ implies that
this $ \mathfrak{s}$ is a splitting of $ \epsilon_1 :\As(\Dd) \ra \Z$.
Since the image of $\mathfrak{s} $ is contained in the central kernel of $ \As(\Dd) \ra \M_2$,
the semi-direct product arising from $\epsilon_{\rm 1} $ is trivial, i.e., 
\begin{equation}\label{dsds} \As (\Dd) \cong \Z \times \Ker(\epsilon_{\rm 1}).
\end{equation}
Moreover, the inclusion $\As(\DDd) \subset \As(\Dd)$ leads to $ \As (\DDd) \cong \Ker(\epsilon_{\rm 1}) $ in a similar way to the case $g \geq 3$.

Next, recall $ H_1( \As(\mathcal{D}_2^{\rm ns})) \cong \Z $
mentioned above, and the basic facts
$H_1(\M_2) \cong \Z/10 $ and $H_2(\M_2) \cong \Z/2 $; see \cite{FM,Kor}.
Then the infraction-restriction
exact sequence for $\mathcal{E} :  \As(\mathcal{D}_2^{\rm ns})\ra \M_2$ can be written as
\begin{equation}\label{OO}H_2(\As(\mathcal{D}_2^{\rm ns})) \lra \Z/2 \stackrel{\delta^* }{\lra} \Ker (\mathcal{E} ) \lra \Z \lra \Z/10 \lra 0, \ \ \ \ {\rm (exact)}.
\end{equation}
Therefore, $\Ker (\mathcal{E} ) $ is either $\Z$ or $\Z \oplus \Z/2$.
Although the image of $ \delta^*$ is mysterious, the quotient $\As(\mathcal{D}_2^{\rm ns})/ \mathrm{Im}(\delta^*)$ is isomorphic to $ \mathcal{T}_2.$

\begin{proof}[Proof of Theorem \ref{thm2}]
Denote the group presented in the statement by $\mathcal{G}$. 
Then, the quotient of $ \mathcal{G}$ subject to $ (c_1 c_2 c_3)^4 c_5^{-2} = 1$ is $\mathcal{M}_2$.
Further, by \eqref{rel}, the quotient map $\mathcal{G} \ra \M_2$ is a central $\Z$-extension.
A diagram chasing of \eqref{OO} easily reveals that
$\kappa_{3\textrm{-}\mathrm{chain }}$ is equal to $e_{c_5}e_{c_4} e_{c_3} e_{c_2} e_{\gamma_1} e_{\gamma_1}e_{\gamma_2} e_{\gamma_3} e_{\gamma_4} e_{\gamma_5} $ contained in the center $\Ker( \mathcal{E})/\mathrm{Im} \delta^* $.
Moreover, we should notice that $A:= [e_{\tau_{\gamma_5}}e_{\tau_{\gamma_4}}e_{\tau_{\gamma_3}}e_{\tau_{\gamma_2}} e_{\tau_{\gamma_1}} e_{\tau_{\gamma_1}}e_{\tau_{\gamma_2}}e_{\tau_{\gamma_3}}e_{\tau_{\gamma_4}} e_{\tau_{\gamma_5}},e_{\tau_{\gamma_5}} $ is zero in $\As(\mathcal{D}_2^{\rm ns})/ \mathrm{Im}(\delta^*)$. Indeed, although $A$ is contained in the center $\Ker( \mathcal{E})/\mathrm{Im} \delta^*  \cong \Z$ by \eqref{rel}, 
$\epsilon_1(A)=0$ and \eqref{OO} imply $ A=0$.
In summary, as before, 
the correspondence $ c_i \mapsto e_{\gamma_i}$ defines an epimorphism between central $\Z$-extensions $ \mathcal{G} $ and $\As(\DDd )$ over $\mathcal{M}_g$, which
is an isomorphism. 
Hence, the above result $ \As(\DDd ) \cong \mathcal{T}_{2} $ immediately leads to the conclusion.
\end{proof}
\begin{rem}\label{thm33} From the discussion, the kernel of $\As(\mathcal{D}_2)/{\mathrm{Im}\delta^* } \ra \mathcal{M}_2$ is generated by $  (e_{c_1 } e_{c_2} )^{6}e_{\tau_{c_{\rm sep}}}^{-1}$ and $ \kappa_{3\textrm{-} {\rm chain}}$.
\end{rem}

\subsection{An application to Lefschetz fibrations}\label{s4}
Finally, we conclude this paper by discussing
right $ \mathcal{T}_{g,0}$-modules and tuples $(z_1, \dots, z_m) \in (\D)^m$ with $z_1 \cdots z_m =1_G$.
Recall that the product $e_{z_1} \cdots e_{z_m} \in \As(\D ) $ lies in the center $\Ker(\mathcal{E})$.
Furthermore, in the study of Lefschetz fibration invariants,
it is important to verify whether the identity $e_{z_1} \cdots e_{z_m}= \mathrm{id}_M$ holds or not (see \cite[\S 3.2]{Nos2} for the details).

To do so in an easy way,
we will show that
the identity can be established by the central elements $ \kappa_{3\textrm{-}\mathrm{chain }} $, $\kappa_{\rm lantern} $ and the signature of 4-dimensional Lefschetz fibrations:
(This is also a criterion for which the quantum representation is useful for some Lefschetz fibration invariants; see \cite{Nos2} for details. )
Precisely,
\begin{prop}\label{dse}
Let us regard a right $ \mathcal{T}_{g,0}$-module $M$ as 
an $\As(\D )$-module via the isomorphism $\As(\D) \cong \mathcal{T}_{g,0} \times \Z$.
Denote the associated map $\As (\D) \ra \mathrm{End}(M)$ by $\rho $.
\begin{enumerate}[(I)]
\item Let $g \geq 3$. For
any tuple $(z_1, \dots, z_m) \in (\D )^m$ with $ z_1\cdots z_m=1$,
the product $\rho (e_{z_1} \cdots e_{z_m})$ is equal to
\begin{equation}\label{kji} \rho ( \kappa_{3\textrm{-}\mathrm{chain }} )^{ ( \sigma_{\mathbf{z}} + m )/4} \rho ( \kappa_{\rm lantern} )^{( 5 \sigma_{\mathbf{z}} +5 m )/2 -m_{\rm ns } } \rho ( \mathcal{L}_1)^{n_1} \cdots \rho ( \mathcal{L}_{[\frac{g}{2}]})^{n_{[\frac{g}{2}]}}  \in\mathrm{End}(M).
\end{equation}
Here, $n_k:= \# \{ z _i \in \mathcal{D}_g^{(k)}\} $, 
and the notation $ \sigma_{\mathbf{z}} , \ m_{\rm ns} \in \mathbb{N} $ 
are the same as Theorem \ref{ENresult}.

In particular, if the right hand side is the identity, 
then $\rho ( e_{z_1} \cdots e_{z_m}) = \mathrm{id }_M$.
\item If $g=2 $, then the identity
$$ \rho (e_{z_1} \cdots e_{z_m}) =  \rho (( e_{c_1 } e_{c_2} )^{6}e_{\tau{c_{\rm sep}}}^{-1} )^{m_{\rm ns}-m } \rho ( \kappa_{3\textrm{-}\mathrm{chain }} )^{(\sigma_{\mathbf{z}}  -7m +7m_{\rm ns} )/6} \in\mathrm{End}(M)$$
holds for any tuple $ \mathbf{z}= (z_1, \dots, z_m) \in (\D)^m$ satisfying $ z_1\cdots z_m=1 $.
\end{enumerate}
\end{prop}

\begin{proof}
It follows from the proof of Theorem \ref{thm3} that the kernel of $  \As (\D ) \ra \M_g$ is generated by $(\kappa_{3 \textrm{-} \mathrm{chain }}\kappa_{\rm lantern }^{10})$ and the lantern relations $ , \kappa_{\rm lantern } $, $\mathcal{L}_1 , \dots,  \mathcal{L}_{[\frac{g}{2}]} $.
Since $e_{z_1} \cdots e_{z_m} $ is contained in the kernel, 
there exist $N_C, \ N_L \in \Z$ for which the following holds:
\begin{equation}\label{sss1} e_{z_1} \cdots e_{z_m}  = (\kappa_{3 \textrm{-} \mathrm{chain }}\kappa_{\rm lantern }^{10}) ^{N_C} (\kappa_{\rm lantern }) ^{N_L} 
(\mathcal{L}_1)^{n_1} \cdots  ( \mathcal{L}_{[\frac{g}{2}]})^{n_{[\frac{g}{2}]}} 
\in \As ( \D) \end{equation}

First, we show (I).
Note $ m_{\rm ns } =\epsilon_{\rm ns}(e_{z_1} \cdots e_{z_m} ) = - N_L \in \Z $.
Furthermore, recall from Theorem \ref{ENresult}
that the homomorphism $I_g$ satisfies
$I_g ( \kappa_{3\textrm{-}\mathrm{chain }} ) = -6 $ and $I_g( \kappa_{\rm lantern}) = I_g(\mathcal{L}_k ) = -1 $; 
The former statement in Theorem \ref{ENresult} yields 
$$ \sigma_{\mathbf{z}} = I_g(e_{z_1} \cdots e_{z_m} )= -6 N_C + 10 N_C -N_L -n_1-\cdots - n_{[\frac{g}{2}]}  = 4 N_C -m, $$
leading to the solution
$N_C= ( \sigma_{\mathbf{z}} + m )/4 $.
Hence, applying $\rho$ to \eqref{sss1} implies the desired equality \eqref{kji} as claimed.

(II) Finally, we deal with the case $ g =2.$ 
By Remark \ref{thm33}, the kernel of $\As(\mathcal{D}_2) \ra \M_2$ is 
generated by $  (e_{c_1 } e_{c_2} )^{6}e_{\tau{c_{\rm sep}}}^{-1}$ and $ \kappa_{3\textrm{-} {\rm chain}}$.
Hence, we have $ e_{z_1} \cdots e_{z_m}  = ((e_{c_1 } e_{c_2} )^{6}e_{\tau{c_{\rm sep}}}^{-1})^{N_1}(\kappa_{3\textrm{-} {\rm chain}})^{N_2} $. 
Then, $m-m_{\rm ns}= \epsilon_{1}(e_{z_1} \cdots e_{z_m} ) = - N_1 $.
Noting from \cite[Proposition 3.9]{EN} that $ I_g(  (e_{c_1 } e_{c_2} )^{6}e_{\tau{c_{\rm sep}}}^{-1})= -7$, we have $  \sigma_{\mathbf{z}}  = I_g(e_{z_1} \cdots e_{z_m} )=  -7 N_1 - 6 N_2 $
Consequently, the solution $N_2= (\sigma_{\mathbf{z}}  - 7 m + 7m_{\rm ns} )/6$ yields the required equality.
\end{proof} 

\subsection*{Acknowledgments}
The author expresses his gratitude to Sylvain Gervais,
Gregor Masbaum, and Takuya Sakasai for useful comments on the mapping class group.
I also grateful to the referee for carefully reading this paper and giving valuable comments. 

\vskip 1pc

\normalsize
 
DEPARTMENT OF
MATHEMATICS
TOKYO
INSTITUTE OF
TECHNOLOGY
2-12-1
OOKAYAMA
, MEGURO-KU TOKYO
152-8551 JAPAN


\end{document}